\documentclass{amsart}
\usepackage{amsmath}
\usepackage{amssymb,amsmath,amscd}
\usepackage{bm}
\usepackage{amsfonts}
\usepackage{epsf}
\usepackage{graphicx}
\usepackage{multirow}
\usepackage{algorithm}
\usepackage{algpseudocode}
\input epsf%

\usepackage{subfigure}
\usepackage{stmaryrd}

\usepackage{soul} 
\usepackage{color, xcolor}

\soulregister\cite7
\soulregister\eqref7
\soulregister\ref7

\DeclareMathOperator*{\dv}{\operatorname{div}}
\DeclareMathOperator*{\tr}{\operatorname{tr}}
\DeclareMathOperator*{\Span}{\operatorname{span}}
\DeclareMathOperator*{\R}{\mathbb{R}}
\DeclareMathOperator*{\Sym}{\mathbb{S}}

\DeclareMathOperator*{\T}{\mathcal{T}}
\DeclareMathOperator*{\TT}{\mathbb{T}}
\DeclareMathOperator*{\E}{\mathcal{E}}
\DeclareMathOperator*{\V}{\mathcal{V}}
\DeclareMathOperator*{\F}{\mathcal{F}}

\newcommand{\dx}[1]{\mathrm{d}#1}
\newcommand{\norm}[1]{\left\lVert#1\right\rVert}
\newcommand{\abs}[1]{\left|#1\right|}
\newcommand{\ip}[2]{\left(#1\right)_{#2}}

\newcommand{\trace}[2]{\left.#1\right|_{#2}}

\newtheorem{thm}{Theorem}[section]
\newtheorem{lemma}[thm]{Lemma}
\newtheorem{assumption}[thm]{Assumption}
\newtheorem{definition}[thm]{Definition}
\newtheorem{remark}[thm]{Remark}

\numberwithin{equation}{section}

\long\def\comment#1{}

\begin{document}

	\title[]{
		A new mixed finite element for the linear elasticity problem in 3D
	}
	
	
	\author {Jun Hu}
	\address{LMAM and School of Mathematical Sciences, Peking University,
		Beijing 100871, P. R. China.\\ Chongqing Research Institute of Big Data, Peking University, Chongqing 401332, P. R. China. hujun@math.pku.edu.cn}
	
	\author{Rui Ma}
	\address{Beijing Institute of Technology,
		Beijing 100081, P. R. China. rui.ma@bit.edu.cn}
	
	\author{Yuanxun Sun}
	\address{LMAM and School of Mathematical Sciences, Peking University,
		Beijing 100871, P. R. China. 1901112048@pku.edu.cn}

	\thanks{The authors were supported by NSFC
		project 12288101}

	\begin{abstract}
		
		\vskip 15pt
		
		This paper constructs the first mixed finite element for the linear elasticity problem in 3D using $P_3$ polynomials for the stress and discontinuous $P_2$ polynomials for the displacement on tetrahedral meshes under some mild mesh conditions. The degrees of freedom of the stress space as well as the corresponding nodal basis are established by characterizing a space of some piecewise constant symmetric matrices on a patch around each edge. Macro-element techniques are used to define a stable interpolation to prove the discrete inf-sup condition. Optimal convergence is obtained theoretically.

		\vskip 15pt
		
		
		\vskip 15pt
		
	\end{abstract}
	\maketitle
	
	\section{Introduction}
	
	It is a challenge to design stable discretizations for the linear elasiticity equations based on the Hellinger-Reissner variational principle, due to the additional symmetry constraint on the stress tensor. For the 3D problem, the first mixed finite element on tetrahedral meshes was proposed in \cite{Adams} as the lowest order generalization of the two dimensional symmetric mixed elements in \cite{Arnold-Winther-conforming}. This element was extended in \cite{Arnold-Awanou-Winther} to higher order cases, where the displacement is approximated by discontinuous $P_{k-1}$ polynomials, and the stress is approximated by $P_{k+2}$ polynomials whose divergence is $P_{k-1}$  with $k\geqslant2$. Later, Hu and Zhang designed symmetric mixed finite elements  using $P_k$ polynomials for the stress with $k\geqslant4$ \cite{HuZhang2015tre}, see \cite{HuZhang2014a} for analogous elements in 2D and \cite{Hu2015trianghigh} in any dimension. Since there are not sufficient degrees of freedom (DoFs) on faces, the analysis of the discrete inf-sup conditions in \cite{Adams,Arnold-Awanou-Winther,HuZhang2015tre,HuZhang2015trianglow} needs  $P_4$ polynomials in stress spaces to define some stable commuting interpolations. Lower order mixed elements merely  using $P_k$ ($k< 4$) polynomials  for the stress  were constructed on macro-element meshes, such as $k\geqslant 2$ for Alfeld splits  \cite{Alfeld} and $k\geqslant 1$ for Worsey-Farin splits \cite{Worsey},  and each macro-element therein consists of four and twelve elements, respectively. Other attempts to lower the polynomial order include  nonconforming finite element methods \cite{Arnold-nonconforming,CaiYe,Guzman2011,Huma2018,IP2017} and stabilized methods \cite{stabilized}. Interested readers can refer  to \cite{AwanouRec,Hu1} for mixed finite elements on cubic meshes and  \cite{HuMaPrism} for those on triangular prism meshes. 
	

	
This paper proposes the first  mixed finite element using $P_3$ polynomials for   the stress without any higher order polynomials on tetrahedral meshes under some mild mesh conditions. The displacement space is the space of discontinuous  $P_2$ polynomials. The newly constructed stress element is $H(\dv)$-conforming and continuous at vertices. The DoFs of this stress element cannot be defined in the Ciarlet's convention. Motivated by the $C^1$ continuous finite elements in 2D in \cite{Nodal,nodalbasis}, this paper establishes the DoFs  of the new stress element and the corresponding basis functions by geometry analysis at edges.   The parity of the number of elements in a patch  around an interior edge and the singularity of the edge play  an essential role and determine  the choices of some face DoFs.  The main ingredient of the analysis is to introduces a space $\TT_e$ consisting of some piecewise constant symmetric matrices on the edge patch $\omega_e$ of an edge $e$. Such a space can be characterized by the normal-normal components  and normal-tangential components of matrices  on the faces sharing $e$. More precisely,  for a boundary edge $e$ as well as an interior singular edge $e$, the matrices in $\TT_e$ can be uniquely determined by the normal-normal component on each face plus the normal-tangential component on one face. For an interior edge $e$ with odd numbers of elements in $\omega_e$, the matrices in $\TT_e$ can be uniquely determined by the normal-normal component on each face. While for an interior non-singular edge $e$ with even elements in $\omega_e$, the normal-normal components of matrices in $\TT_e$  are linearly dependent, and the matrices in $\TT_e$ can be uniquely determined by the normal-normal components on all faces except one plus the normal-tangential component on one face. The corresponding basis of $\TT_e$ can be computed by explicit expressions.   With the characterization of $\TT_e$, some DoFs on faces can be obtained for the new stress space, which are similar to the second order derivative DoFs on edges in \cite{nodalbasis}.  The multiplications of the basis of $\TT_e$ with scalar Lagrange basis functions lead to the corresponding basis for the new stress  space. The other DoFs of the new stress space are analogous to the DoFs of the stress spaces in \cite{ChenHuang2022,HuZhang2015tre}.
Analogous ideas been applied in two dimensions to relax the continuity of the stress at vertices in \cite{Gong,HuMa}. Instead of presenting the details of the nodal basis, a hybridized method was considered in  \cite{Gong}. In \cite{HuMa}, this idea was used to deal with non-consistent traction boundary conditions. Besides, some estimates of the basis of $\TT_e$ are obtained in this paper under the  mesh conditions presented in Assumption \ref{ass}. 

This paper adopts the two-step method in \cite{Hu2015trianghigh,HuZhang2015tre,HuZhang2015trianglow} and some macro-element techniques to analyze the discrete inf-sup condition of the new mixed element.  The two-step method 
therein requires  to define a stable commuting interpolation. However, for an interior non-singular edge $e$ with even number of elements in $\omega_e$, the normal-normal components on faces are linearly dependent for matrices in $\TT_e$. As a result, it is difficult to construct a stable commuting interpolation as in \cite[Lemma 3.1]{Hu2015trianghigh}. To circumvent this, under some mild mesh conditions, this paper proposes  proper linear combinations  of  the basis functions of the stress space for each face. These combinations are supported on some macro-elements and are exactly the basis functions with respect to the constant and linear moments of the normal-normal component  of the stress on each face. This and the estimates of the basis of $\TT_e$ enable to define a stable interpolation for the new stress space using $P_3$ polynomials as the shape function space to deal with the rigid motion as in \cite[Lemma 3.1]{Hu2015trianghigh}.  The discrete inf-sup condition  is then established by using the two-step method in \cite{Hu2015trianghigh,HuZhang2015tre,HuZhang2015trianglow}. Some error estimates are provided as well.

	The rest of the paper is organized as follows. Section 2 presents the mixed formulation, notation for triangulations and some mild mesh conditions. A symmetric-matrix space on an edge patch is introduced and characterized in this section. Section 3 proposes a new mixed element  using $P_3$ polynomials for the stress and discontinuous $P_2$ polynomials for the displacement. A unisolvent set of DoFs is given for the stress space. The discrete stability and optimal convergence results are proved. Section 4 summarizes the results and introduces possible future work.
	
	\section{Preliminaries and notation}
	
	This section presents the mixed formulation of the linear elasticity equations in 3D and notation for triangulations, and prepares some results of a space of some piecewise constant matrices for later use.

	\subsection{Mixed formulation}
	Let $\Omega\subset\R^3$ be a simply-connected bounded polyhedral domain in three dimensions. Based on the Hellinger-Reissner principle, the  linear elasticity
	problem within a stress-displacement ($\sigma$-$u$) form reads: given $f\in V$, find $(\sigma,u)\in \Sigma\times V$  such that
	\begin{equation}
	\left\{
	\begin{aligned}
	(A\sigma,\tau)_{\Omega}+(\dv\tau,u)_{\Omega}&=0,\\
	(\dv\sigma,v)_{\Omega}&=(f,v)_{\Omega}
	\end{aligned}
	\right.
	\label{continuousP}
	\end{equation}
	holds for any $(\tau,v)\in\Sigma\times V$, where $\Sigma:=H({\rm div},\Omega;\mathbb {S})$ is the symmetric stress field and $V:=L^2(\Omega;\mathbb{R}^3)$ is the displacement field, the compliance tensor $A(x):\Sym\rightarrow\Sym$ is bounded and symmetric positive
	definite uniformly for $x\in\Omega$, with $\Sym:=\R^{3\times3}_{\text{sym}}$ being the set of symmetric matrices. For the homogeneous isotropic case the compliance tensor is given by $A\tau=\frac{1}{2\mu}\left(\tau-\frac{\lambda}{2\mu+3\lambda}\left(\tr\tau\right) I\right)$, where $\mu>0$, $\lambda\geqslant0$ are the Lam\'{e} constants.

	The space $H(\dv,\Omega;\Sym)$ is defined by
	\begin{equation*}
	H(\dv,\Omega;\Sym):=\{\tau\in L^2(\Omega;\Sym):\ \dv\tau\in L^2(\Omega;\R\nolimits^3)\},
	\end{equation*}
	equipped with the norm
	\begin{equation*}
	\norm{\tau}_{H(\dv,\Omega)}^2:=\norm{\tau}_{L^2(\Omega)}^2+
	\norm{\dv\tau}_{L^2(\Omega)}^2.
	\end{equation*}
	The problem \eqref{continuousP} has a unique solution in $\Sigma\times V$ \cite[Chapter 9.1.1]{Boffi-Brezzi-Fortin2013}. The discretization of \eqref{continuousP} can avoid the locking phenomenon and provide a direct approximation for the stress.

	Throughout this paper, let $H^k(\omega;X)$ denote the Sobolev space consisting of functions with domain $\omega\subset\R^3$, taking values in the finite-dimensional vector space $X$ ($\R$, $\R^3$ or $\Sym$), and with all derivatives of order at most $k$ square-integrable. Let $\norm{\cdot}_{k,\omega}$ and $\abs{\cdot}_{k,\omega}$ be the norm and the seminorm of $H^k(\omega;X)$, and let $\left(\cdot,\cdot\right)_{\omega}$ denote the inner product of $L^2(\omega)$. When $\omega=\Omega$, the norm and the seminorm are simply denoted by $\norm{\cdot}_k$ and $\abs{\cdot}_k$. Let $P_k(\omega;X)$ denote the space of polynomials of degree at most $k$, taking value in the space $X$.

	\subsection{Some notation and mesh conditions}
	\label{sec:mesh}

	Let $\T_h$ be a shape-regular triangulation of $\Omega$. Let $h_K=\operatorname{diam}(K)$ be the diameter of element $K$ and let $h=\max_{K\in\T_h}h_K$ indicate the maximal mesh size. For an element $K$, the sets of all the vertices, edges, and faces of $K$ are denoted by $\V(K)$, $\E(K)$ and $\F(K)$ respectively. Let $\V(\T_h)$, $\E(\T_h)$ and $\F(\T_h)$ denote the sets of all the vertices, edges, and faces of $\T_h$. Let $\E^b(\T_h)$ and $\E^\circ(\T_h)$ denote the sets of all the boundary and interior edges of $\T_h$, respectively. For an edge $e$, the  edge patch sharing the edge $e$ is denoted by $\omega_e=\bigcup_{K\in\T_h:\ e\in\E(K)}K$. See Figure \ref{fig:macro} for the edge patch for an interior edge $e$. Let $m_e$ be the number of elements in $\omega_e$. Let $K_i$ ($1\leqslant i\leqslant m_e$) denote the elements in $\omega_e$ and let $F_i$ ($1\leqslant i\leqslant m_e+1$) denote the faces containing $e$ so that $F_i\subset\partial K_i$ and $F_{i+1}\subset\partial K_i$ for $i=1,2,\cdots,m_e$. If $e$ is an interior edge, then $F_1=F_{m_e+1}$, \emph{i.e.} the elements $K_1$ and $K_{m_e}$ are adjacent. Let $\bm{n}_i$ be the unit normal vector of $F_i$ and $\bm{t}$ be the tangential vector along the edge $e$, and let $\bm{t}_i=\bm{n}_i\times\bm{t}$ be the other tangential vector of the face $F_i$. Let $\theta_i$ be the dihedral angle between $F_{i}$ and $F_{i+1}$ ($1\leqslant i\leqslant m_e$). For the sake of simplicity, when $e$ is an interior edge, the subscripts for the angles and the faces are understood in the sense of$\mod m_e$. When $e$ is a boundary edge, $F_1$ and $F_{m_e+1}$ are the boundary faces. For convenience, define the nodal patch $\Omega_{x_0}=\bigcup_{K\in\T_h:\ x_0\in\V(K)}K$ for a vertex $x_0$, the face patch $\Omega_{F}=\bigcup_{K'\in\T_h:\ F\cap K'\neq\emptyset}K'$ for a face $F$, and the element patch $\Omega_K=\bigcup_{K'\in\T_h:\ K\cap K'\neq\emptyset}K'$ for an element $K$.

		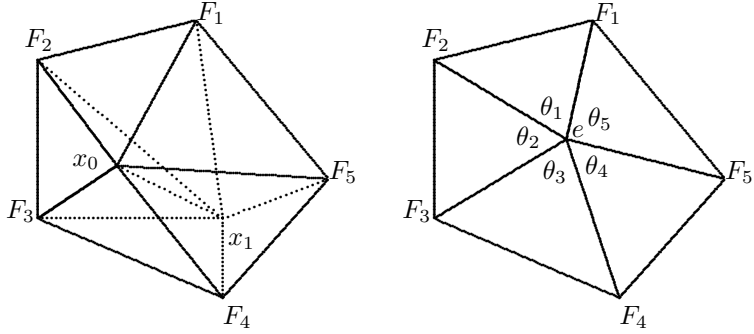
\begin{figure}[h]
			\centering
				\begin{picture}(200,150)(40,-22)
				\def\lb{\circle*{0.8}}
				\def\lc{\vrule width1.2pt height1.2pt}
				\def\la{\circle*{0.4}} 
				

				\put(-12,  20){$F_3$} 
				\put(60,  95){$F_1$} 
				\put(-5,  85){$F_2$} 
				\put(70,  -20){$F_4$} 
				\put(110,  35){$F_5$} 
				\put(13,  40){$x_0$} 
				\put(72,  10){$x_1$} 

				\multiput(0,  80)(.6,  .15){100}{\la} 
				\multiput(00,  80)(0,  -.6){100}{\la} 
				\multiput(00,  20)(.7,  -.3){100}{\la} 
				\multiput(70, -10)(.4,  .45){100}{\la} 
				\multiput(60,  95)(.5,  -.6){100}{\la} 
				
				\multiput(30,  40)(.3,  .55){100}{\la} 
				\multiput(30,  40)(-.3,  .4){100}{\la} 
				\multiput(30,  40)(-.3,  -.2){100}{\la} 
				\multiput(30,  40)(.4,  -.5){100}{\la} 
				\multiput(30,  40)(.8,  -.05){100}{\la} 

				\multiput(70,  20)(-.25,  1.875){40}{\la} 
				\multiput(70,  20)(-1.75,  1.5){40}{\la} 
				\multiput(70,  20)(-1.75,  0){40}{\la} 
				\multiput(70,  20)(0,  -1.5){20}{\la} 
				\multiput(70,  20)(2,  .75){20}{\la} 
				
				\multiput(30,  40)(2,  -1){20}{\la} 
				
				
				\put(138,  20){$F_3$} 
				\put(210,  95){$F_1$} 
				\put(145,  85){$F_2$} 
				\put(220,  -20){$F_4$} 
				\put(260,  35){$F_5$} 
				\put(202,  51){$e$} 
				
				\put(190,  59){$\theta_1$} 
				\put(181,  48){$\theta_2$}
				\put(191,  35){$\theta_3$}
				\put(207,  37){$\theta_4$}
				\put(208,  55){$\theta_5$}

				\multiput(150,  80)(.6,  .15){100}{\la} 
				\multiput(150,  80)(0,  -.6){100}{\la} 
				\multiput(150,  20)(.7,  -.3){100}{\la} 
				\multiput(220, -10)(.4,  .45){100}{\la} 
				\multiput(210,  95)(.5,  -.6){100}{\la} 
				
				\multiput(200,  50)(.1,  .45){100}{\la} 
				\multiput(200,  50)(-.5,  .3){100}{\la} 
				\multiput(200,  50)(-.5,  -.3){100}{\la} 
				\multiput(200,  50)(.2,  -.6){100}{\la} 
				\multiput(200,  50)(.6,  -.15){100}{\la} 
				
				\end{picture}
				\caption{The left figure displays an edge patch for the discrete stress space. This patch consists of $m_e=5$ elements that share the common edge $e$ with vertices $x_0$ and $x_1$. The right figure shows the patch projected onto a plane.}
				\label{fig:macro}
		\end{figure}

	\begin{definition}
		An edge $e\in\E^\circ(\T_h)$ is called a singular edge, if $m_e=4$ and $\theta_1+\theta_2=\theta_2+\theta_3=\pi$. See Figure \ref{fig:singular}.
	\end{definition}

	\begin{figure}[h]\setlength\unitlength{1.5pt}
		\centering
		\begin{picture}(100,80)(10,-22)
		\def\lb{\circle*{0.8}}
		\def\lc{\vrule width1.2pt height1.2pt}
		\def\la{\circle*{0.4}} 
		
		\put(0,  -6){$F_3$} 
		\put(60,  45){$F_1$} 
		\put(20,  43){$F_2$} 
		\put(120,  -10){$F_4$} 
		\put(40,  25){$e$} 
		
		\put(37,  34){$\theta_1$} 
		\put(29,  28){$\theta_2$}
		\put(37,  20){$\theta_3$} 
		\put(45,  29){$\theta_4$} 
		
		\multiput(20,  40)(.4,  .05){100}{\la} 
		\multiput(0,  0)(.1,  .2){200}{\la} 
		\multiput(0,  0)(.3,  -.025){400}{\la} 
		\multiput(120,  -10)(-.3,  .275){200}{\la} 
		
		\multiput(40,  30)(.2,  .15){100}{\la} 
		\multiput(40,  30)(-.2,  .1){100}{\la} 
		\multiput(40,  30)(-.2,  -.15){200}{\la} 
		\multiput(40,  30)(.4,  -.2){200}{\la} 
		
		\end{picture}
		\caption{Singular edge $e$.}
		\label{fig:singular}
	\end{figure}
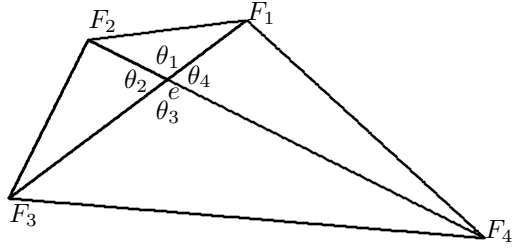
	
	When $e$ is an interior non-singular edge with even $m_e$, the face $F_1$ is chosen so that $\theta_1+\theta_{m_e}\neq\pi$.

	This paper assumes that the following conditions on $\T_h$ hold true.
	\begin{assumption}	\label{ass}
		There exists a positive constant $\kappa>0$ such that
		\begin{enumerate}
			\item For any edge $e$ of $\T_h$, $\kappa<\theta_j<\pi-\kappa$ for $1\leqslant j\leqslant m_e$.
			\item If an interior edge $e$ is not singular and $m_e$ is even, then $\theta_j+\theta_{j+1}<\pi-\kappa$ for $1\leqslant j\leqslant m_e$.
		\end{enumerate}
	\end{assumption}

	The first conditions in Assumption \ref{ass} are the minimum angle conditions. The second conditions in Assumption \ref{ass} are some mild mesh conditions. For instance, it is not difficult to check that if the initial mesh satisfies the mesh conditions, then its uniform refinement satisfies the mesh conditions as well.
	\subsection{Piecewise constant matrices}
	\label{sec:constantT}

	This subsection introduces a matrix space associated with the edge patch $\omega_e$ for an edge $e$ and characterizes the space by some normal-normal components and normal-tangential components on faces. More precisely, define $\TT_e$ to be the space of piecewise constant symmetric matrices  on $\omega_e$ such that the tangential components along $e$ vanish and the normal components are continuous across the interior faces, \emph{i.e.}
	\begin{equation*}
	\begin{aligned}
		\TT\nolimits_e=&\{T:\ \trace{T}{K_j}\in\Sym,\ (\trace{T}{K_j})\bm{t}=0,\\
		&\text{$T\bm{n}_j$ is continuous across $F_j$},\ j=1,2,\cdots,m_e\}.
	\end{aligned}
	\end{equation*}
	For the sake of simplicity, for $T\in\TT_e$, the normal components $\trace{(T\bm{n}_j)}{F_j}$ are denoted by $T\bm{n}_j$, the normal-normal component  $\trace{(\bm{n}_j^TT\bm{n}_j)}{F_j}$ is denoted by $\bm{n}_j^TT\bm{n}_j$, and the normal-tangential component $\trace{(\bm{n}_j^TT\bm{t}_j)}{F_j}$ is denoted by $\bm{n}_j^TT\bm{t}_j$. The normal-normal component and the normal tangential component on each face can be viewed as  linear functionals on $\TT_e$. For $1\leqslant j\leqslant m_e$, since $\trace{(\bm{n}_{j+1}^TT\bm{n}_{j})}{F_j}=\trace{(\bm{n}_j^TT\bm{n}_{j+1})}{F_j}=\trace{(\bm{n}_j^TT\bm{n}_{j+1})}{F_{j+1}}$, this component is denoted by $\bm{n}_j^TT\bm{n}_{j+1}$, and is a linear functional as well.

	The following lemma is a simple but technical result. It shall be later used in the construction of the basis of $\TT_e$ and the estimation of their norms. The matrix norm is chosen to be the Frobenius norm, denoted by $\norm{\cdot}_f$. The vector norm is denoted by $\norm{\cdot}_s$, for $1\leqslant s\leqslant\infty$.
	
	\begin{lemma}
		\label{T-bound}
		Let $\bm{t}_1$ and $\bm{t}_2$ be two non-colinear vectors in $\R^3$, and $\bm{t}=\bm{t}_1\times\bm{t}_2$, $\bm{n}_1=\bm{t}\times\bm{t}_1$, $\bm{n}_2=\bm{t}\times\bm{t}_2$. Let $\theta$ be the angle between $\bm{t}_1$ and $\bm{t}_2$. Given two vectors $\bm{l}_1$ and $\bm{l}_2$ satisfying $\bm{l}_1^T\bm{t}=\bm{l}_2^T\bm{t}=0$, and the compatible condition $\bm{l}_1^T\bm{n}_2=\bm{l}_2^T\bm{n}_1$, there exists a unique $T\in\Span\{\bm{n}_1\bm{n}_1^T,\bm{n}_2\bm{n}_2^T,(\bm{n}_1\bm{n}_2^T+\bm{n}_2\bm{n}_1^T)\}\subset\Sym$ such that $T\bm{n}_1=\bm{l}_1$, $T\bm{n}_2=\bm{l}_2$, and
		\begin{equation*}
			\norm{T}_f^2\leqslant\left(\frac{3}{2}+2\cot^2\theta\right)\left(\norm{\bm{l}_1}_2^2+\norm{\bm{l}_2}_2^2\right).
		\end{equation*}
	\end{lemma}
	\begin{proof}
		The lemma is proved by construction. Let $\bm{l}_1=c\bm{n}_2+a\bm{t}_2$ and $\bm{l}_2=c\bm{n}_1+b\bm{t}_1$, with coefficients $a,b$ and $c=\bm{l}_1^T\bm{n}_2=\bm{l}_2^T\bm{n}_1$. Expand $T$ as $T=A\bm{n}_1\bm{n}_1^T+B\bm{t}_1\bm{t}_1^T+C\left(\bm{n}_1\bm{t}_1^T+\bm{t}_1\bm{n}_1^T\right)$. The fact that $T\bm{n}_1=\bm{l}_1$ implies $A=\bm{l}_1^T\bm{n}_1=c\cos\theta+a\sin\theta$, $C=\bm{l}_1^T\bm{t}_1=-c\sin\theta+a\cos\theta$, and the fact that $T\bm{n}_2=\bm{l}_2$ implies
		\begin{equation*}
			B=\frac{1}{\sin\theta}(C\cos\theta-b)=-\frac{b}{\sin\theta}-c\cos\theta+a\frac{\cos^2\theta}{\sin\theta}.
		\end{equation*}
		
		Therefore,
		\begin{equation*}
			\begin{aligned}
				\norm{T}_f^2&=A^2+B^2+2C^2=a^2+b^2+c^2+\left(c-\cot\theta\left(a-b\right)\right)^2\\
				&\leqslant a^2+b^2+c^2+\left(1+4\cot^2\theta\right)\left(\frac{a^2+b^2}{2}+c^2\right)\\
				&\leqslant\left(\frac{3}{2}+2\cot^2\theta\right)(a^2+b^2+2c^2)=\left(\frac{3}{2}+2\cot^2\theta\right)\left(\norm{\bm{l}_1}_2^2+\norm{\bm{l}_2}_2^2\right).
			\end{aligned}
		\end{equation*}
This concludes the proof.
	\end{proof}

	The following lemma characterize $\TT_e$ and plays a crucial role in the construction of a basis of $\TT_e$. Since the airy function $J=\begin{bmatrix}
		\frac{\partial^2}{\partial x_2^2}&-\frac{\partial^2}{\partial x_1\partial x_2}\\
		-\frac{\partial^2}{\partial x_1\partial x_2}&\frac{\partial^2}{\partial x_1^2}
		\end{bmatrix}$ maps $H^2(\Omega)$ to $H(\dv,\Omega;\Sym)$ in 2D \cite{Arnold-Winther-conforming}, the construction of the second order derivative DoFs in \cite{Nodal,nodalbasis} for the $C^1$ continuous finite elements in 2D can be generalized to the current case. The parity of the number of elements $m_e$ in the patch around an interior edge $e$   determines the choices of the basis of $\TT_e$. Especially for an interior non-singular edge $e$ with even $m_e$, the constraint of the tangential second order partial derivatives \cite[(32)]{Nodal} can be used to derive a constraint for the normal-normal components of matrices in $\TT_e$. For completeness, this paper gives a detailed proof.

	\begin{lemma}  Let $\omega_e$ be the edge patch of  an edge $e$ consisting of $m_e$ elements. It holds that:
		
		\begin{enumerate}
			\item $e$ is an interior non-singular edge with even $m_e$: For any $T\in\TT_e$, it holds
			\begin{equation}
			\sum_{j=1}^{m_e}(-1)^j\left(\cot\theta_{j-1}+\cot\theta_j\right)\bm{n}_j^TT\bm{n}_j=0.
			\label{even-constraint}
			\end{equation}
			On the contrary, for any $\gamma_j\in\R$, $1\leqslant j\leqslant m_e$ satisfying $\sum_{j=1}^{m_e}(-1)^j(\cot\theta_{j-1}+\cot\theta_j)\gamma_j=0$, there exists $T\in\TT_e$ such that $\bm{n}_j^TT\bm{n}_j=\gamma_j$. Further, such $T$ is unique if $\bm{n}_1^TT\bm{t}_1=0$. Especially for $\gamma_j=0$ with $1\leqslant j\leqslant m_e$, there exists a unique $T\in\TT_e$ such that $\bm{n}_j^TT\bm{n}_j=\gamma_{j}=0$, and $\bm{n}_1^TT\bm{t}_1=1$.

			\item $e$ is an interior edge with odd $m_e$: For any $\gamma_j\in\R$, $1\leqslant j\leqslant m_e$, there exists a unique $T\in\TT_e$ such that $\bm{n}_j^TT\bm{n}_j=\gamma_j$.
			
			\item $e$ is a singular edge: For any $\gamma_j\in\R$, $1\leqslant j\leqslant4$, there exists $T\in\TT_e$ such that $\bm{n}_j^TT\bm{n}_j=\gamma_j$. Further, such $T$ is unique if $\bm{n}_1^TT\bm{t}_1=0$. Especially for $\gamma_{j}=0$ with $1\leqslant j\leqslant4$, there exists a unique $T\in\TT_e$ such that $\bm{n}_j^TT\bm{n}_j=\gamma_j=0$, and $\bm{n}_1^TT\bm{t}_1=1$.
			
			\item $e$ is a boundary edge: For any $\gamma_j\in\R$, $1\leqslant j\leqslant m_e+1$, there exists $T\in\TT_e$ such that $\bm{n}_j^TT\bm{n}_j=\gamma_j$. Further, such $T$ is unique if $\bm{n}_1^TT\bm{t}_1=0$. Especially for $\gamma_j=0$ with $1\leqslant j\leqslant m_e+1$, there exists a unique $T\in\TT_e$ such that $\bm{n}_j^TT\bm{n}_j=\gamma_{j}=0$, and $\bm{n}_1^TT\bm{t}_1=1$.

		\end{enumerate}

		\label{normal-normal}
	\end{lemma}
	\begin{proof}
		
		Given $T\in\TT_e$, assume $T\bm{n}_j=\gamma_j\bm{n}_j+\alpha_j\bm{t}_j$ for $j=1,2,\cdots,m_e\ (m_e+1 \text{ for $e\in\E^{b}(\T_h)$ })$. Note that $\gamma_{j}=\bm{n}_j^TT\bm{n}_j$ and $\alpha_{j}=\bm{n}_j^TT\bm{t}_j$. For each element $K_j$ with $1\leqslant j\leqslant m_e$, the compatible condition in Lemma \ref{T-bound} reads $\bm{n}_{j+1}^T\left(T\bm{n}_j\right)=\bm{n}_{j}^T\left(T\bm{n}_{j+1}\right)$, i.e. 
		\begin{equation*}
		\gamma_j\cos\theta_j-\alpha_j\sin\theta_j=\gamma_{j+1}\cos\theta_j+\alpha_{j+1}\sin\theta_j,
		\end{equation*}
		or equivalently,
		\begin{equation}
		(\gamma_j-\gamma_{j+1})\cot\theta_j=\alpha_{j}+\alpha_{j+1}.
		\label{compatible}
		\end{equation}
		The proof is divided into four cases.
		
		\begin{enumerate}
			\item \textbf{Case 1:} $e$ is an interior non-singular edge with even $m_e$. 
			A summation of the multiplication of \eqref{compatible} with $(-1)^j$ shows
			\begin{equation*}
			\begin{aligned}
				0&=\sum_{i=1}^{m_e}(-1)^j\left(\alpha_{j}+\alpha_{j+1}\right)=\sum_{j=1}^{m_e}(-1)^j(\gamma_j-\gamma_{j+1})\cot\theta_j\\
				&=\sum_{j=1}^{m_e}(-1)^j\gamma_j\left(\cot\theta_{j-1}+\cot\theta_{j}\right).
			\end{aligned}
			\end{equation*}
			This proves \eqref{even-constraint}. On the contrary, given $\gamma_{j}$ satisfying \eqref{even-constraint}, Lemma \ref{T-bound} guarantees the existence of the desired $T$ if there exists $\alpha_{j}$, $1\leqslant j\leqslant m_e$ such that \eqref{compatible} holds. Actually, for fixed $\alpha_{1}\in\R$, $\alpha_2,\cdots,\alpha_{m_e}$ can be determined uniquely by solving \eqref{compatible} sequentially. Such $T$ is unique if $\alpha_{1}$ is set to be zero. Especially, for $\gamma_1=\cdots=\gamma_{m_e}=0$, \eqref{compatible} leads to $\alpha_{1}=-\alpha_{2}=\cdots=-\alpha_{m_e}$. By setting $\alpha_{1}=1$, all the other $\alpha_{j}$ are determined.

			\item \textbf{Case 2:} $e$ is an interior edge with odd $m_e$. Given any $\gamma_{j}$, similar to the previous case, it suffices to show that there exists $\alpha_{j}$, $1\leqslant j\leqslant m_e$ such that \eqref{compatible} holds. These equations are linearly independent, and $\alpha_j$ can be solved uniquely as follows
			\begin{equation}
				\alpha_j=\frac{1}{2}\sum_{i=j}^{j+m_e-1}(-1)^{i-j}\left(\gamma_i-\gamma_{i+1}\right)\cot\theta_i,\quad 1\leqslant j\leqslant m_e.
				\label{eq: odd_alpha}
			\end{equation}

			\item \textbf{Case 3:} $e$ is a singular edge. Given any $\gamma_{j}$, the mesh condition $\theta_1=\pi-\theta_2=\theta_3=\pi-\theta_4$ shows that \eqref{compatible} read
			\begin{equation}
				\left\{
				\begin{aligned}
					\left(\gamma_1-\gamma_2\right)\cot\theta_1&=\alpha_1+\alpha_2,\\
					-\left(\gamma_2-\gamma_3\right)\cot\theta_1&=\alpha_2+\alpha_3,\\
					\left(\gamma_3-\gamma_4\right)\cot\theta_1&=\alpha_3+\alpha_4,\\
					-\left(\gamma_4-\gamma_1\right)\cot\theta_1&=\alpha_4+\alpha_1.
				\end{aligned}
				\right.
				\label{eq: singular_alpha}
			\end{equation}
			For fixed $\alpha_{1}\in\R$, $\alpha_2,\alpha_{3},\alpha_{4}$ can be determined uniquely by solving \eqref{eq: singular_alpha} sequentially. Such $T$ is unique if $\alpha_{1}$ is set to be zero. Especially, for $\gamma_1=\cdots=\gamma_{4}=0$, \eqref{eq: singular_alpha} leads to $\alpha_{1}=-\alpha_{2}=\alpha_{3}=-\alpha_{4}$. By setting $\alpha_{1}=1$, all the other $\alpha_{j}$ are determined.
			
			\item \textbf{Case 4: } $e$ is a boundary edge. Given any $\gamma_{j}$ and fixed $\alpha_{1}\in\R$, $\alpha_2,\cdots,\alpha_{m_e+1}$ can be determined uniquely by solving \eqref{compatible} sequentially. Such $T$ is unique if $\alpha_{1}$ is set to be zero. Especially, for $\gamma_1=\cdots=\gamma_{m_e+1}=0$, \eqref{compatible} leads to $\alpha_{1}=-\alpha_{2}=\cdots=(-1)^{m_e}\alpha_{m_e+1}$. By setting $\alpha_{1}=1$, all the other $\alpha_{j}$ are determined.
		\end{enumerate}
	
	The above discussions conclude the proof.
	\end{proof}

	\begin{remark}
		The proof of Lemma \ref{normal-normal} shows that by choosing proper $\gamma_{j}$, one can construct a basis of $\TT_e$. In particular,
		\begin{equation}
			\dim\TT\nolimits_e=\begin{cases}
			m_e,&\text{if $e$ is an interior non-singular edge},\\
			m_e+1,&\text{if $e$ is a singular edge},\\
			m_e+2,&\text{if $e$ is a boundary edge}.
			\end{cases}
			\label{dim_Te}
		\end{equation}
		Assumption \ref{ass} is not required in this lemma.
	\end{remark}

As shown in Lemma \ref{normal-normal}, the normal-normal components on all faces of each matrix in $\TT_e$ are  linearly dependent for an interior non-singular edge $e$ with even $m_e$. This will bring difficulties in the analysis of the discrete inf-sup condition of the mixed element. To circumvent this, the following lemma introduces some matrices in $\TT_e$ that have normal-normal components of the same sign on two adjacent interior faces, under the mesh conditions in Assumption \ref{ass}. Those matrices and the boundness of their norms will enable to construct a stable interpolation operator for proving the discrete inf-sup condition in next section. 		Analogous arguments in the norm estimates have been used in \cite{Gong} to measure the vertex singularity. For convenience, define the norm $\norm{T}:=\max_{1\leqslant j\leqslant m_e}\norm{\trace{T}{K_j}}_f$ for $T\in\TT_e$.

	\begin{lemma}\label{T-exist}Suppose $\T_h$ satisfies Assumption \ref{ass}. Let $\omega_e$ be the edge patch for  an edge $e$ consisting of $m_e$ elements. It holds that:
	\begin{enumerate}
	\item If $e$ is an interior non-singular edge with even $m_e$, then there exists $T\in\TT_e$ and some positive constants $C,\ C_1$ and $C_2$ independent of $e$ and $m_e$, such that
			\begin{equation}
		\begin{aligned}
				&\bm{n}_{1}^TT\bm{n}_{1}=1,\ \bm{n}_{2}^TT\bm{n}_{2}\geqslant 0,
				\bm{n}_{j}^TT\bm{n}_{j}=0,\ \text{for $j=3,4,\cdots m_e$},\\
				&C_1\kappa^{2}\leqslant\bm{n}_{2}^TT\bm{n}_{2}\leqslant C_2\kappa^{-2},\quad \norm{T}\leqslant C\kappa^{-2}.
		\end{aligned}
			\label{condition}
		\end{equation}
	\item  Otherwise, there exists $T\in \TT_e$ and a positive constant $C$ independent of $e$ and $m_e$ such that
			\begin{equation}\label{conditionOther}	\begin{aligned}
				&\bm{n}_{1}^TT\bm{n}_{1}=1,\bm{n}_{j}^TT\bm{n}_{j}=0,\ \text{for $j=2,3,\cdots m_e$}(m_e+1 \text{ for $e\in\E\nolimits^{\partial}(\T\nolimits_h)$ })\\
&\ \norm{T}\leqslant C\kappa^{-2}.	\end{aligned}
			\end{equation}
	\end{enumerate}
	\end{lemma}
	\begin{proof}
		The proof follows some arguments in Lemma \ref{normal-normal}.
		
		\textbf{Step 1 considers an interior non-singular edge $e$ with even $m_e$.} The mesh conditions in Assumption \ref{ass} read $\kappa<\theta_j<\pi-\kappa$ and $\theta_{j}+\theta_{j+1}<\pi-\kappa$, with $1\leqslant j\leqslant m_e$. The identity \eqref{even-constraint} in Lemma \ref{normal-normal} implies that the desired $T$ satisfies
		\begin{equation*}
		\bm{n}_2^TT\bm{n}_2=\frac{\cot\theta_1+\cot\theta_{m_e}}{\cot\theta_1+\cot\theta_2}=\frac{\sin\theta_2\sin\left(\theta_1+\theta_{m_e}\right)}{\sin\theta_{m_e}\sin\left(\theta_1+\theta_2\right)}:=\beta>0
		\end{equation*}
		with $C_1\kappa^{2}\leqslant\beta\leqslant C_2\kappa^{-2}$ for some constants $C_1$ and $C_2$.
		
		Given $\gamma_{1}=1$, $\gamma_{2}=\beta$ and $\gamma_{3}=\cdots=\gamma_{m_e}=0$ in Lemma \ref{normal-normal}, there exists $T\in\TT_e$ satisfying $T\bm{n}_1=\bm{n}_1+\alpha_1\bm{t}_1$, $T\bm{n}_2=\beta\bm{n}_2+\alpha_2\bm{t}_2$, and $T\bm{n}_j=\alpha_j\bm{t}_j$, $3\leqslant j\leqslant m_e$. Here $\alpha_{j}$ satisfies the compatible conditions in \eqref{compatible} as follows
		\begin{equation*}
		\left\{
		\begin{aligned}
		\alpha_1-\alpha_3&=-\cot\theta_{m_e},\\
		\alpha_1+\alpha_2&=(1-\beta)\cot\theta_1,\\
		\alpha_3=-\alpha_4&=\cdots=(-1)^{m_e-3}\alpha_{m_e}.
		\end{aligned}
		\right.
		\end{equation*}
		Especially, one can choose $\alpha_1=-\cot\theta_{m_e}$, $\alpha_2=(1-\beta)\cot\theta_{1}+\cot\theta_{m_e}$, $\alpha_3=\cdots=\alpha_{m_e}=0$. Lemma \ref{T-bound} lead that $T$ is supported on $K_{m_e}\bigcup K_1\bigcup K_2$ and
		\begin{align}
		\begin{cases}
		\norm{\trace{T}{K_1}}_f^2\leqslant C\left(1+\cot^2\theta_{1}\right)(1+\alpha_{1}^2+\beta^2+\alpha_{2}^2),\\
		\norm{\trace{T}{K_2}}_f^2\leqslant C\left(1+\cot^2\theta_{2}\right)(\beta^2+\alpha_{2}^2+\alpha_{3}^2),\\
		\norm{\trace{T}{K_{m_e}}}_f^2\leqslant C\left(1+\cot^2\theta_{m_e}\right)(1+\alpha_{1}^2).
		\end{cases}
		\label{est: T}
		\end{align}
		Recall $\kappa<\theta_j<\pi-\kappa$ and $\theta_{j}+\theta_{j+1}<\pi-\kappa$, with $1\leqslant j\leqslant m_e$. It remains to estimate the squares of $\beta\cot\theta_1$, $\beta\cot^2\theta_1$, $\beta\cot\theta_2$, $\beta\cot\theta_1\cot\theta_2$, and the other terms appearing in the above inequalities are bounded by $C\kappa^{-4}$.
		Note that $\frac{\sin(\theta)}{\sin(\theta_1+\theta)}$ is positive and strictly increasing on $\left[\kappa,\pi-\theta_1-\kappa\right]$, and that $\frac{\sin\kappa}{\sin\theta_1}\leqslant1$. This and elementary calculations show
		\begin{align*}
			\abs{\beta\cot\theta_1}&=\frac{\sin\theta_2}{\sin\left(\theta_1+\theta_2\right)}\frac{\sin\left(\theta_1+\theta_{m_e}\right)}{\sin\theta_{m_e}}\abs{\cot\theta_1}\leqslant\frac{\sin^2(\theta_1+\kappa)}{\sin^2\kappa}\abs{\cot\theta_1}\\
			&=\frac{\sin(\theta_1+\kappa)}{\sin^2\kappa}\abs{\cos\theta_1}\left(\cos\kappa+\cos\theta_1\frac{\sin\kappa}{\sin\theta_1}\right)\leqslant C\kappa^{-2}.
		\end{align*}
		Similarly,
		\begin{align*}
			\abs{\beta\cot^2\theta_1}&\leqslant\frac{\cos^2\theta_1}{\sin^2\kappa}\left(\cos\kappa+\cos\theta_1\frac{\sin\kappa}{\sin\theta_1}\right)^2\leqslant C\kappa^{-2},\\
			\abs{\beta\cot\theta_2}&=\frac{\abs{\cos\theta_2}}{\sin\left(\theta_1+\theta_2\right)}\frac{\sin\left(\theta_1+\theta_{m_e}\right)}{\sin\theta_{m_e}}\leqslant C\kappa^{-2},\\
			\abs{\beta\cot\theta_1\cot\theta_2}&=\frac{\abs{\cos\theta_2}}{\sin\left(\theta_1+\theta_2\right)}\frac{\sin\left(\theta_1+\theta_{m_e}\right)}{\sin\theta_{m_e}}\abs{\cot\theta_1}\leqslant C\kappa^{-2}.
		\end{align*}
		The above estimates into \eqref{est: T} lead to
		\begin{equation*}
			\norm{\trace{T}{K_j}}_f^2\leqslant C\kappa^{-4},\ j=1,2,m_e.
		\end{equation*}

			\textbf{Step 2 considers an interior edge $e$ with odd $m_e$.} Given $\gamma_{1}=1$, and $\gamma_{2}=\cdots=\gamma_{m_e}=0$ in Lemma \ref{normal-normal}, there exists $T\in\TT_e$ satisfying $T\bm{n}_1=\bm{n}_1+\alpha_1\bm{t}_1$ and $T\bm{n}_j=\alpha_j\bm{t}_j$ for $2\leqslant j\leqslant m_e$. Here $\alpha_{j}$ is given in \eqref{eq: odd_alpha}: $\alpha_1=\frac{\cot\theta_1-\cot\theta_{m_e}}{2}$, $\alpha_2=\frac{\cot\theta_1+\cot\theta_{m_e}}{2}$, and $\alpha_2=-\alpha_3=\cdots=(-1)^{m_e-2}\alpha_{m_e}=-\alpha_{m_e}$. Assumption \ref{ass} yields $\alpha_1^2+\alpha_2^2\leqslant C\kappa^{-2}$. This and Lemma \ref{T-bound} lead to
			\begin{equation*}
				\norm{T}^2\leqslant C\left(1+\kappa^{-2}\right)\max\{1+\alpha_1^2+\alpha_2^2,2\alpha_2^2\}\leqslant C\kappa^{-4}.
			\end{equation*}

			\textbf{Step 3 considers a singular edge $e$.} Given $\gamma_{1}=1$, and $\gamma_{2}=\gamma_{3}=\gamma_{4}=0$ in Lemma \ref{normal-normal}, there exists $T\in\TT_e$ satisfying $T\bm{n}_1=\bm{n}_1+\alpha_1\bm{t}_1$ and $T\bm{n}_j=\alpha_j\bm{t}_j$ for $j=2,3,4$. Here $\alpha_{j}$ can be chosen as $\alpha_1=\cot\theta_1$, $\alpha_2=\alpha_3=\alpha_4=0$ by \eqref{eq: singular_alpha}. This, Lemma \ref{T-bound} and Assumption \ref{ass} lead to
			\begin{equation*}
				\norm{T}^2\leqslant C\left(1+\kappa^{-2}\right)(1+\alpha_1^2)\leqslant C\kappa^{-4}.
			\end{equation*}

			\textbf{Step 4 considers a boundary edge $e$.} Given $\gamma_{1}=1$, and $\gamma_{2}=\cdots=\gamma_{m_e+1}=0$ in Lemma \ref{normal-normal}, there exists $T\in\TT_e$ satisfying $T\bm{n}_1=\bm{n}_1+\alpha_1\bm{t}_1$ and $T\bm{n}_j=\alpha_j\bm{t}_j$ for $2\leqslant j\leqslant m_e+1$. Here $\alpha_{j}$ can be chosen as $\alpha_1=\cot\theta_1$, $\alpha_2=\cdots=\alpha_{m_e+1}=0$ by the compatible conditions \eqref{compatible}. This, Lemma \ref{T-bound} and Assumption \ref{ass} lead to
				\begin{equation*}
				\norm{T}^2\leqslant C\left(1+\kappa^{-2}\right)(1+\alpha_1^2)\leqslant C\kappa^{-4}.
				\end{equation*}

				The combination of Step 1-4 concludes the proof.
	\end{proof}

	\section{A new mixed finite element}
	
	This section introduces a new mixed finite element using $P_3$ polynomials for the stress  and discontinuous $P_2$ polynomials for the displacement. A unisolvent set of DoFs is established for the newly constructed discrete stress space. Then the discrete inf-sup condition is established through macro-element techniques. Finally, the error estimate and the improved estimate are proved.
	
	\subsection{The mixed finite element spaces}
	
	The discrete stress space is the space of piecewise $P_3(K;\Sym)$ functions continuous at vertices:
	\begin{equation}
	\begin{aligned}
		\Sigma_{h}=\{\tau\in H(\dv,\Omega;\Sym):\ \trace{\tau}{K}\in P_3(K;\Sym),\ \forall K\in\T\nolimits_h,\\
		\text{$\tau$ is continuous at vertices}\}.
	\end{aligned}
	\label{stress_space}
	\end{equation}
	The discrete displacement space is
	\begin{equation}	\label{disp_space}
	V_h=\{v_h\in L^2(\Omega;\R\nolimits^3):\ \trace{v_h}{K}\in P_2(K;\R\nolimits^3),\ \forall K\in\T\nolimits_h\}.
	\end{equation}
	
	Similar as the $H(\dv)$-conforming symmetric spaces in 2D in \cite{Gong}, there are no locally defined
DoFs on single element for $\Sigma_h$. Motivated by the characterization of the space $\TT_e$ from Section \ref{sec:constantT}, the DoFs will be given on edge patches. Given $F\in\F\left(\T_h\right)$, let $\bm{n}$ be the unit normal vector of $F$. Define the projection onto the tangent space of $F$ by $\Pi_F(v)=\left(\bm{n}\times v\right)\times\bm{n}$. Let $ND_1(F)$ be the first kind of N\'{e}d\'{e}lec element of face $F$ \cite{Boffi-Brezzi-Fortin2013}, i.e.
	\begin{equation*}
		ND_1(F)=\{v\in\left(P_1(F)\right)^3+\left(P_1(F)\right)^3\times\bm{x}:\ v\cdot\bm{n}=0\}.
	\end{equation*}
	Given an edge $e$, recall the edge patch $\omega_e$ and the related notation from Section \ref{sec:mesh}.
	
	The degrees of freedom read as follows
	\begin{equation}
	\left\{
		\begin{aligned}
			&\tau(x),\quad\forall x\in \V(\T\nolimits_h),\\
			&\ip{\bm{n}^T\tau\bm{n},q}{F},\quad\forall q\in P_0(F),\ \forall F\in\F(\T\nolimits_h),\\
			&\ip{\Pi_F(\tau\bm{n}),v}{F},\quad\forall v\in ND_1(F),\ \forall F\in\F(\T\nolimits_h),\\
			&\ip{\tau,\xi}{K},\quad\forall\xi\in P_1(K;\Sym),\ \forall K\in\T\nolimits_h.
		\end{aligned}
	\right.
	\label{DOF1}
	\end{equation}

		\begin{equation}
			\begin{aligned}
				\ip{\trace{\left(\bm{n}_1^T\tau\bm{t}_1\right)}{F_1},q}{e},& \ip{\trace{\left(\bm{n}_i^T\tau\bm{n}_i\right)}{F_i},q}{e},\\
				&\forall q\in P_1(e),\ 1\leqslant i\leqslant m_e+1,\ \forall e\in\E\nolimits^{b}(\T\nolimits_h),
			\end{aligned}
			\label{DOF2-1}
		\end{equation}

		\begin{equation}
			\ip{\trace{\left(\bm{n}_i^T\tau\bm{n}_i\right)}{F_i},q}{e},\ \forall q\in P_1(e),\ 1\leqslant i\leqslant m_e,\ \forall e\in\E\nolimits^{\circ}(\T\nolimits_h) \text{ with odd $m_e$},
			\label{DOF: odd}
		\end{equation}
		
		\begin{equation}
		\begin{aligned}
			\ip{\trace{\left(\bm{n}_1^T\tau\bm{t}_1\right)}{F_1},q}{e},& \ip{\trace{\left(\bm{n}_i^T\tau\bm{n}_i\right)}{F_i},q}{e},\ \forall q\in P_1(e),\ 2\leqslant i\leqslant m_e,\\
			&\forall \text{ non-singular $e\in\E\nolimits^{\circ}(\T\nolimits_h)$ with even  $m_e$},
		\end{aligned}
		\label{DOF: even}
		\end{equation}
		
		\begin{equation}
		\begin{aligned}
			\ip{\trace{\left(\bm{n}_1^T\tau\bm{t}_1\right)}{F_1},q}{e},& \ip{\trace{\left(\bm{n}_i^T\tau\bm{n}_i\right)}{F_i},q}{e},\\
			&\forall q\in P_1(e),\ 1\leqslant i\leqslant4,\ \forall \text{ singular $e\in\E\nolimits^{\circ}(\T\nolimits_h)$}.
		\end{aligned}
			\label{DOF2}
		\end{equation}

		\begin{remark}
					The DoFs in \eqref{DOF1} are analogous to those of the stress elements in \cite{ChenHuang2022}. Recall the $H(\dv)$ bubble function space $$\Sigma_{K,3,b}=\{\tau\in P_3(K;\Sym):\ \tau\bm{n}=0\text{ on $\partial K$}\}$$ defined in \cite[Section 2.2]{Hu2015trianghigh}. Note that the fourth DoFs in \eqref{DOF1} replace the moments with $\Sigma_{K,3,b}$ originally given in \cite{Hu2015trianghigh}. The DoFs in \eqref{DOF2-1}-\eqref{DOF2} are proposed according to Lemma \ref{normal-normal} and are similar to the second order derivative DoFs for the $C^1$ conforming element in \cite{nodalbasis}. The basis functions with respect to the DoFs in \eqref{DOF1} can be established by following the construction of analogous basis functions for $P_3$ in \cite{ChenHuang2022,HuZhang2015tre}. The basis functions with respect to the DoFs in \eqref{DOF2-1}-\eqref{DOF2} can be established by the multiplication of the basis of $\TT_e$ with scalar Lagrange basis functions. The matrix space $\TT_e$ is characterized in Lemma \ref{normal-normal} and its basis can be easily obtained. The details are omitted here.
		\end{remark}

	The following theorem shows that the set of DoFs in \eqref{DOF1}-\eqref{DOF2} is unisolvent.
	
	\begin{thm}[Unisolvence]
		The set of degrees of freedom given by (\ref{DOF1})-(\ref{DOF2}) is unisolvent for $\Sigma_{h}$.
		\label{Thm: unisolvence}
	\end{thm}
	\begin{proof}

			 \textbf{Step 1 counts the dimension.} For an edge $e$, let $\lambda_{i}$ ($i=1,2$) be the barycentric coordinates associated with the two vertices of $e$, and define the edge bubble function space $\Sigma_e$ to be:
			\begin{equation*}
			\begin{aligned}
				\Sigma_e&=\left\{\tau\in\Sigma_{h}:\ \trace{\tau}{\omega_e}=\lambda_1\lambda_2(a\lambda_1+b\lambda_2)T,\ a,b\in\R,\ T\in\TT\nolimits_e,\right.\\
				&\quad\quad\left.\text{$\tau$ vanishes in $\Omega\backslash\omega_e$}\right\}.
			\end{aligned}
			\end{equation*}
			Since the functions in $\Sigma_e$ vanish on all the edges except for $e$, the sum of $\Sigma_e$ over $e\in\E(\T_h)$ is direct, and $\dim\Sigma_e=2\dim\TT\nolimits_e$.

			Let $\Sigma_{CH}$ be the analogous discrete stress space given in \cite[Theorem 4.13]{ChenHuang2022} for $k=3$, i.e.
			\begin{equation*}
			\begin{aligned}
				\Sigma_{CH}=\{\tau\in H(\dv,\Omega;\Sym):\ &\text{$\trace{\tau}{K}\in P_3(K;\Sym)$ for $\forall K\in\T\nolimits_h$},\\
				 &\text{$\tau$ is continuous at vertices,}\\
				 &\text{$\bm{n}_i^T\tau\bm{n}_j$ is continuous at each edge $e$, $1\leq i\leq j\leq 2$}\}.
			\end{aligned}
			\end{equation*}
			The DoFs for $\Sigma_{CH}$ consist of the DoFs given in \eqref{DOF1} and the following edge DoFs instead of those in \eqref{DOF2-1}-\eqref{DOF2}:
			\begin{equation}
				\ip{\bm{n}_i^T\tau\bm{n}_j,q}{e},\quad\forall q\in P_1(e),\ \forall e\in\E(\T\nolimits_h),\ i,j=1,2.
				\label{DOF_CH}
			\end{equation}
			Here $\bm{n}_1$ and $\bm{n}_2$ are two normal vectors of $e$.

			Notice that
			\begin{equation}
				\Sigma_{h}\supseteq\{\tau\in\Sigma_{CH}:\ \text{$\tau$ vanishes on DoFs in \eqref{DOF_CH}}\}+\bigoplus_{e\in\E(\T_h)}\Sigma_e.
				\label{space}
			\end{equation}
			Since the nodal basis functions in $\Sigma_{CH}$ associated with DoFs in \eqref{DOF1} have vanishing DoFs in \eqref{DOF_CH}, it follows that the sum in the right hand side of \eqref{space} is also direct, and
			\begin{equation*}
				\dim\Sigma_{h}\geqslant6\#\V(\T\nolimits_h)+9\#\F(\T\nolimits_h)+24\#\T\nolimits_h+2\sum_{e\in\E(\T_h)}\dim\TT\nolimits_e.
			\end{equation*}
			Here $\#\cdot$ denotes the cardinal number of the set and $\dim\TT_e$ is given in \eqref{dim_Te}. The right hand side is exactly the number of the DoFs of the new stress space.


			 \textbf{Step 2 proves that if $\tau\in\Sigma_{h}$ vanishes at all DoFs, then $\tau=0$.} The vanishing first DoFs in \eqref{DOF1} show that $\tau$ vanishes at all vertices of $\T_h$.  Recall that $i=m_e+1$ is understood as $i=1$ if $e$ is an interior edge. For the cases in \eqref{DOF2-1}, \eqref{DOF: odd} and \eqref{DOF2}, the vanishing DoFs lead to $\trace{\left(\trace{\left(\bm{n}_i^T\tau\bm{n}_i\right)}{F_i}\right)}{e}=0$ with $1\leqslant i\leqslant m_e+1$. For the case in \eqref{DOF: even}, note that $F_1$ is chosen in Section 2.2 so that $\cot\theta_1+\cot\theta_{m_e}\neq0$, the combination with \eqref{even-constraint} shows $\trace{\left(\trace{\left(\bm{n}_i^T\tau\bm{n}_i\right)}{F_i}\right)}{e}=0$ with $1\leqslant i\leqslant m_e$. In all cases, this implies $\trace{\left(\trace{\left(\tau\bm{n}_i\right)}{F_i}\right)}{e}=\alpha_i\bm{t}_i+\beta_i\bm{t}$ for some $\alpha_i, \beta_i\in P_3(e)$. Similar to the arguments in Lemma \ref{normal-normal}, the compatible conditions show $\alpha_1=-\alpha_2=\cdots=(-1)^{m_e-1}\alpha_{m_e}=(-1)^{m_e}\alpha_{m_e+1}$. If $e\in\E^{\circ}(\T_h)$ and $m_e$ is odd, this leads to $\alpha_i=0$ for $1\leqslant i\leqslant m_e$. Otherwise, the first DoF in \eqref{DOF2-1}, \eqref{DOF: even} and \eqref{DOF2} shows $\alpha_1=0$. The previous arguments lead to $\trace{\left(\trace{\left(\tau\bm{n}_i\right)}{F_i}\right)}{e}=\beta_i\bm{t}$, for $1\leqslant i\leqslant m_e+1$.


			 Since $e$ is arbitrary, this fact combined with the second vanishing DoFs in \eqref{DOF1} leads to $\trace{\bm{n}^T\tau\bm{n}}{F}=0$ for any face $F$. As in \cite[Lemma 4.5]{ChenHuang2022}, this combined with the third vanishing DoFs in \eqref{DOF1} again leads to $\trace{\tau\bm{n}}{F}=0$ for any face $F$. This means for any $K\in\T_h$, $\trace{\tau}{K}$ is an $H(\dv)$ bubble function in $\Sigma_{K,3,b}$. The fourth DoFs in \eqref{DOF1} deal with the remaining bubbles and lead to $\tau=0$ \cite{ChenHuang2022,Hu2015trianghigh}. This concludes the proof.
	\end{proof}
	
	\subsection{Discrete inf-sup condition}
	
	The discrete inf-sup condition is analyzed by macro-element techniques and the two-step method \cite[Lemma 3.1, Theorem 3.1]{Hu2015trianghigh}. The first step of the two-step method requires a modified Scott-Zhang interpolation function $\tau_h$ to satisfy
	\begin{equation}
		\int_{\partial K}\tau_h\bm{n}\cdot v\dx{S}=\int_{\partial K}\tau\bm{n}\cdot v\dx{S},\ \forall v\in RM(K),\ \forall K\in\T\nolimits_h,
		\label{modifiedSZ}
	\end{equation}
for some $\tau\in H^1(\Omega;\Sym)$ with $\dv\tau=v_h\in V_h$. Due to the decomposition $v=(v\cdot\bm{n})\bm{n}$ $+\Pi_F(v)$, \eqref{modifiedSZ} can be guaranteed if it holds
	\begin{align}
		\hspace{-1mm}\int_{F}\bm{n}^T\tau_h\bm{n}\left(v\cdot\bm{n}\right)\dx{S}=\int_{F}\bm{n}^T\tau\bm{n}\left(v\cdot\bm{n}\right)\dx{S},\ \forall v\in RM(K),\ \forall F\subset\partial K,
	\label{eq-normal-normal}\\
	\hspace{-1mm}\int_{F}\Pi_F\left(\tau_h\bm{n}\right)\cdot\Pi_F(v)\dx{S}=\int_{F}\Pi_F\left(\tau\bm{n}\right)\cdot\Pi_F(v)\dx{S},\ \forall v\in RM(K),\ \forall F\subset\partial K.\hspace{-1mm}
	\label{eq-tangential-normal}
	\end{align}
	
Firstly, a $\tau_h$ is constructed to satisfy \eqref{eq-normal-normal}. Note that $v\cdot\bm{n}\in P_1(F)$ and there are not enough DoFs in \eqref{DOF1} on faces. Unlike the stress spaces in \cite{Adams,Hu2015trianghigh,HuZhang2015tre} which involve $P_4$ polynomials, in the current case, one cannot construct face bubble functions to satisfy \eqref{eq-normal-normal}. To circumvent this difficulty, the key idea in the following lemma is to utilize Lemma \ref{T-exist} and propose  proper linear combinations  of  the basis functions  in three neighbouring edge patches for each face. The  combinations are exactly the basis functions with respect to the constant and linear moments of the normal-normal component  of the stress on each face. A modified Scott-Zhang interpolation is then constructed locally on macro-elements, which consist of some edge patches.

	\begin{lemma}
		Let $I_h:H^1(\Omega;\Sym)\rightarrow\Sigma_h\cap H^1(\Omega;\Sym)$ be the Scott-Zhang interpolation operator \cite{ScottZhang}. For $\tau\in H^1(\Omega;\Sym)$, there exists $\delta_h^1\in\Sigma_h$ such that for each face $F$,
		\begin{equation}
			\int_{F}\bm{n}^T\delta_h^1\bm{n}p\dx{S}=\int_{F}\bm{n}^T\left(\tau-I_h\tau\right)\bm{n}p\dx{S},\quad\forall p\in P_1(F),
			\label{normal-normal-eq}
		\end{equation}
		and that
		\begin{equation}
			\sum_{K\in\T_h}\left(\norm{\nabla\delta_h^1}_{0,K}^2+h_K^{-2}\norm{\delta_h^1}_{0,K}^2\right)\leqslant C\abs{\tau}_1^2.
			\label{delta-h1-bound}
		\end{equation}
		Here $C$ is a constant dependent on $\kappa$ in Assumption \ref{ass}.
		\label{Lem: delta_h1}
	\end{lemma}
	\begin{proof}
		The proof is divided into five steps.
		
		\textbf{Step 1 finds $\delta_h^{F,i}\in\Sigma_{h}$ for any face $F$ and $i=1,2,3$ such that}
		\begin{equation}
			\left\{
			\begin{aligned}
			\int_{F}\bm{n}^T\delta_h^{F,i}\bm{n}\cdot q_j\dx{S}&=\delta_{ij}\abs{F},\quad j=1,2,3,\\
			\int_{F'}\bm{n}^T\delta_h^{F,i}\bm{n}\cdot q\dx{S}&=0,\quad\forall q\in P_1(F'),\ F'\neq F.
			\end{aligned}
			\right.
			\label{delta_Fj}
		\end{equation}
		Here $\{q_i\}_{1\leqslant i\leqslant3}$ can be chosen as a Lagrangian basis of $P_1(F)$ and $\abs{F}$ denotes the area of $F$.
		
		Assume that $F$ belongs to an element $K=x_0x_1x_2x_3$, see Figure \ref{fig:3dp3p2}. Let $F_i$ be the face of $K$ opposite to the vertex $x_i$ for $i=0,1,2,3$. Without loss of generality, suppose $F=F_2$. Let $e_{ij}$ be the edge between vertices $x_i$ and $x_j$ for $0\leqslant i<j\leqslant3$. Let $\lambda_i$ be the barycentric coordinates with respect to $x_i$. Let $q_1=\trace{\lambda_0}{F_2}$, $q_2=\trace{\lambda_1}{F_2}$, $q_3=\trace{\lambda_3}{F_2}$.

		\begin{figure}[h]\setlength\unitlength{1.5pt}
			\centering
			\begin{picture}(100,80)(0,-22)
			\def\lb{\circle*{0.8}}
			\def\lc{\vrule width1.2pt height1.2pt}
			\def\la{\circle*{0.4}} 
			
			\put(0,  -5){$x_2$} 
			\put(60,  60){$x_0$} 
			\put(80,  -25){$x_1$} 
			\put(105,  -5){$x_3$} 
			
			\put(55,  -10){$F_0$} 
			\put(40,  10){$F_3$} 
			\put(77,  11){$F_2$} 
			\put(53,  20){$F_1$} 

			\multiput(0,  0)(.2,  -.05){400}{\la} 
			\multiput(0,  0)(.15,  .15){400}{\la} 
			\multiput(0,  0)(1.05,  0){100}{\la} 
			
			\multiput(60,  60)(.1125,  -.15){400}{\la} 
			\multiput(60,  60)(.05,  -.2){400}{\la} 
			\multiput(105,  0)(-.0625,  -.05){400}{\la} 

			\end{picture}
			\caption{Element $K$. The face opposite to the vertex $x_i$ is denoted by $F_i$, $i=0,1,2,3$.}
			\label{fig:3dp3p2}
		\end{figure}
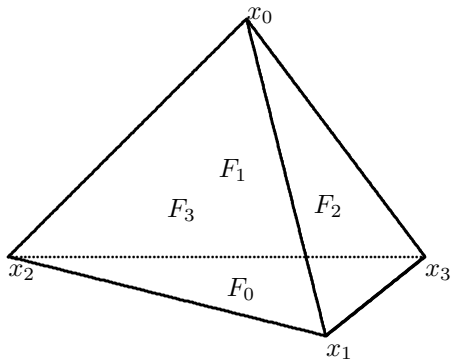

	Lemma \ref{T-exist} shows that there exists $T_1\in\TT_{e_{01}}$ such that $\bm{n}_2^TT_1\bm{n}_2=1$, $\bm{n}_3^TT_1\bm{n}_3\geqslant0$, the normal-normal component of $T_1$ vanishes at the other faces containing $e_{01}$. Select $\delta_1=\lambda_0\lambda_1\left(\lambda_0-1\slash3\right)T_1$. Simple calculations show
	\begin{equation}\label{eq:orthscalar}
	\begin{aligned}
		\int_{F_2}\lambda_0\lambda_1^2\left(\lambda_0-1\slash3\right)\dx{S}=\int_{F_2}\lambda_0\lambda_1\lambda_3\left(\lambda_0-1\slash3\right)\dx{S}=0,\\
		\int_{F_3}\lambda_0\lambda_1^2\left(\lambda_0-1\slash3\right)\dx{S}=\int_{F_3}\lambda_0\lambda_1\lambda_2\left(\lambda_0-1\slash3\right)\dx{S}=0.
	\end{aligned}
	\end{equation}
	The above results lead to
	\begin{equation}
	\begin{aligned}	
		\int_{F_2}\bm{n}_2^T\delta_1\bm{n}_2\cdot\lambda_1\dx{S}=\int_{F_2}\bm{n}_2^T\delta_1\bm{n}_2\cdot\lambda_3\dx{S}=0,\\
		\int_{F_3}\bm{n}_3^T\delta_1\bm{n}_3\cdot\lambda_1\dx{S}=\int_{F_3}\bm{n}_3^T\delta_1\bm{n}_3\cdot\lambda_2\dx{S}=0.
	\end{aligned}
	\label{delta1-vanishing}
	\end{equation}
	
	Since
	\begin{equation}
		\int_{F_3}\bm{n}_3^T\delta_1\bm{n}_3\cdot\lambda_0\dx{S}=\bm{n}_3^TT_1\bm{n}_3\frac{\abs{F_3}}{180}\geqslant0,
		\label{F3-lambda0}
	\end{equation}
	one needs to construct a function $\delta_2$ to cancel out the non-vanishing linear moment of $\bm{n}_3^T\delta_1\bm{n}_3$ on $F_3$. Lemma \ref{T-exist} shows that there exists $T_2\in\TT_{e_{02}}$ such that $\bm{n}_3^TT_2\bm{n}_3=1$, $\bm{n}_1^TT_2\bm{n}_1\geqslant0$, the normal-normal component of $T_2$ vanishes at the other faces containing $e_{02}$. Select $\delta_2=\lambda_0\lambda_2\left(\lambda_0-1\slash3\right)T_2$. Simple calculations show
	\begin{equation}
	\begin{aligned}
		\int_{F_3}\bm{n}_3^T\delta_2\bm{n}_3\cdot\lambda_1\dx{S}&=\int_{F_3}\bm{n}_3^T\delta_2\bm{n}_3\cdot\lambda_2\dx{S}=0,\\
		\int_{F_1}\bm{n}_1^T\delta_2\bm{n}_1\cdot\lambda_2\dx{S}&=\int_{F_1}\bm{n}_1^T\delta_2\bm{n}_1\cdot\lambda_3\dx{S}=0,\\
		\int_{F_3}\bm{n}_3^T\delta_2\bm{n}_3\cdot\lambda_0\dx{S}&=\frac{\abs{F_3}}{180}.
	\end{aligned}
	\label{delta2-vanishing}
	\end{equation}
	The combination of \eqref{delta1-vanishing}-\eqref{delta2-vanishing} shows
	\begin{equation*}
		\int_{F_3}\bm{n}_3^T\big(\delta_1-\left(\bm{n}_3^TT_1\bm{n}_3\right)\delta_2\big)\bm{n}_3\cdot q\dx{S}=0,\ \forall q\in P_1(F_3).
	\end{equation*}

	Similarly, one needs to construct a function $\delta_3$ to cancel out the non-vanishing linear moment of $\bm{n}_1^T\big(\delta_1-\left(\bm{n}_3^TT_1\bm{n}_3\right)\delta_2\big)\bm{n}_1$ on $F_1$. Lemma \ref{T-exist} shows that there exists  $T_3\in\TT_{e_{03}}$ such that $\bm{n}_1^TT_3\bm{n}_1=1$, $\bm{n}_2^TT_3\bm{n}_2\geqslant0$, the normal-normal component of $T_3$ vanishes at the other faces containing $e_{03}$. A choice of $\delta_3=\lambda_0\lambda_3\left(\lambda_0-1\slash3\right)T_3$ leads to
	\begin{equation*}
		\begin{aligned}
			\int_{F_3}\bm{n}_3^T\big(\delta_1-\left(\bm{n}_3^TT_1\bm{n}_3\right)\delta_2+\left(\bm{n}_3^TT_1\bm{n}_3\right)\left(\bm{n}_1^TT_2\bm{n}_1\right)\delta_3\big)\bm{n}_3\cdot q\dx{S}&=0,\ \forall q\in P_1(F_3),\\
			\int_{F_1}\bm{n}_1^T\big(\delta_1-\left(\bm{n}_3^TT_1\bm{n}_3\right)\delta_2+\left(\bm{n}_3^TT_1\bm{n}_3\right)\left(\bm{n}_1^TT_2\bm{n}_1\right)\delta_3\big)\bm{n}_1\cdot q\dx{S}&=0,\ \forall q\in P_1(F_1).
		\end{aligned}
	\end{equation*}
	
	On the other hand, \eqref{eq:orthscalar}  and \eqref{F3-lambda0} imply
	\begin{equation*}
		\begin{aligned}
			&\int_{F_2}\bm{n}_2^T\big(\delta_1-\left(\bm{n}_3^TT_1\bm{n}_3\right)\delta_2+\left(\bm{n}_3^TT_1\bm{n}_3\right)\left(\bm{n}_1^TT_2\bm{n}_1\right)\delta_3\big)\bm{n}_2\cdot\lambda_j\dx{S}\\
			&=\begin{cases}\big(1+\left(\bm{n}_3^TT_1\bm{n}_3\right)\left(\bm{n}_1^TT_2\bm{n}_1\right)\left(\bm{n}_2^TT_3\bm{n}_2\right)\big)\frac{\abs{F_2}}{180} \neq0, &j=0,\\
			0,&j=1,3.
			\end{cases}
		\end{aligned}
	\end{equation*}
	
	By normalization, $\delta_h^{F_2,1}$ in \eqref{delta_Fj} with $F=F_2$ can be chosen as
	\begin{equation}
		\delta_h^{F_2,1}=\frac{180\big(\delta_1-\left(\bm{n}_3^TT_1\bm{n}_3\right)\delta_2+\left(\bm{n}_3^TT_1\bm{n}_3\right)\left(\bm{n}_1^TT_2\bm{n}_1\right)\delta_3\big)}{1+\left(\bm{n}_3^TT_1\bm{n}_3\right)\left(\bm{n}_1^TT_2\bm{n}_1\right)\left(\bm{n}_2^TT_3\bm{n}_2\right)}.
		\label{F2-1}
	\end{equation}
	For $i=2,3$, $\delta_h^{F_2,i}$ can be similarly constructed.
	
	\textbf{Step 2 defines}
	\begin{equation}
		\delta_h^1=\sum_{F'\in\F(\T_h)}\sum_{i=1}^{3}\frac{\int_{F'}\bm{n}^T\left(\tau-I_h\tau\right)\bm{n}\cdot q_i\dx{S}}{\abs{F'}}\delta_h^{F',i}.
		\label{delta_h1}
	\end{equation}
	
	For any face $F$ and $q\in P_1(F)$, let $\{q_j\}_{1\leqslant j\leqslant3}$ be a basis of $P_1(F)$ and expand $q=\sum_{j=1}^{3}\alpha_jq_j$, then
	\begin{equation*}
		\begin{aligned}
			&\int_{F}\bm{n}^T\delta_h^1\bm{n}\cdot q\dx{S}=\int_{F}\left(\sum_{i=1}^{3}\frac{\int_{F}\bm{n}^T\left(\tau-I_h\tau\right)\bm{n}\cdot q_i\dx{S}}{\abs{F}}\bm{n}^T\delta_h^{F,i}\bm{n}\right)\cdot q\dx{S}\\
			=&\sum_{i=1}^{3}\sum_{j=1}^{3}\alpha_j\frac{\int_{F}\bm{n}^T\left(\tau-I_h\tau\right)\bm{n}\cdot q_i\dx{S}}{\abs{F}}\int_{F}\left(\bm{n}^T\delta_h^{F,i}\bm{n}\right)q_j\dx{S}\\
			=&\sum_{i=1}^{3}\alpha_i\int_{F}\bm{n}^T\left(\tau-I_h\tau\right)\bm{n}\cdot q_i\dx{S}=\int_{F}\bm{n}^T\left(\tau-I_h\tau\right)\bm{n}\cdot q\dx{S}.
		\end{aligned}
	\end{equation*}
	This proves \eqref{normal-normal-eq}.
	
	\textbf{Step 3 proves the bound of $\delta_h^{F,i}$.} The bound \eqref{delta-h1-bound} is standard from the error estimate for the Scott-Zhang interpolation operator and the scaling argument. Take  \eqref{F2-1}  for an example. The support of $\delta_h^{F_2,1}$ is $\operatorname{supp}\left(\delta_h^{F_2,1}\right):=\bigcup_{1\leqslant i\leqslant3}\omega_{e_{0i}}$ and this forms a macro-element. Actually $\operatorname{supp}\left(\delta_h^{F_2,1}\right)$ is included in the nodal patch $\Omega_{x_0}$. The triangle inequality for the norm yields 
	\begin{equation}
		\norm{\delta_h^{F_2,1}}_*\leqslant\frac{180\big(\norm{\delta_1}_*+\left(\bm{n}_3^TT_1\bm{n}_3\right)\norm{\delta_2}_*+\left(\bm{n}_3^TT_1\bm{n}_3\right)\left(\bm{n}_1^TT_2\bm{n}_1\right)\norm{\delta_3}_*\big)}{1+\left(\bm{n}_3^TT_1\bm{n}_3\right)\left(\bm{n}_1^TT_2\bm{n}_1\right)\left(\bm{n}_2^TT_3\bm{n}_2\right)}
		\label{norm}
	\end{equation}
with $\norm{\delta}_*=\norm{\delta}_{0,K}+h_K\norm{\nabla\delta}_{0,K}$, for an arbitrary element $K\subset\operatorname{supp}\left(\delta_h^{F_2,1}\right)$. According to Lemma \ref{T-exist}, \eqref{norm} can be estimated by discussing the following cases:
	
	\begin{enumerate}
		\item  If $\left(\bm{n}_3^TT_1\bm{n}_3\right)\left(\bm{n}_1^TT_2\bm{n}_1\right)=0$, then \eqref{norm} and Lemma \ref{T-exist} lead to
		\begin{equation*}
		\norm{\delta_h^{F_2,1}}_*\leqslant C\left(1+\kappa^{-2}\right)\left(\norm{\delta_1}_*+\norm{\delta_2}_*\right)\leqslant C\kappa^{-4}h_K^{3\slash2}.
		\end{equation*}
		
		\item If $\left(\bm{n}_3^TT_1\bm{n}_3\right)\left(\bm{n}_1^TT_2\bm{n}_1\right)\neq0$ and $\bm{n}_2^TT_3\bm{n}_2=0$, then 
		\begin{equation*}
		\norm{\delta_h^{F_2,1}}_*\leqslant C\big(\norm{\delta_1}_*+\left(\bm{n}_3^TT_1\bm{n}_3\right)\norm{\delta_2}_*+\left(\bm{n}_3^TT_1\bm{n}_3\right)\left(\bm{n}_1^TT_2\bm{n}_1\right)\norm{\delta_3}_*\big)\leqslant C\kappa^{-6}h_K^{3\slash2}.
		\end{equation*}
		
		\item  If $\left(\bm{n}_3^TT_1\bm{n}_3\right)\left(\bm{n}_1^TT_2\bm{n}_1\right)\left(\bm{n}_2^TT_3\bm{n}_2\right)\neq0$, then
		\begin{equation*}
		\norm{\delta_h^{F_2,1}}_*\leqslant C\big(\norm{\delta_1}_*+\left(\bm{n}_3^TT_1\bm{n}_3\right)\norm{\delta_2}_*+\left(\bm{n}_2^TT_3\bm{n}_2\right)^{-1}\norm{\delta_3}_*\big)\leqslant C\kappa^{-4}h_K^{3\slash2}.
		\end{equation*}
		
	\end{enumerate}

	For any $K\subset\operatorname{supp}\left(\delta_h^{F,i}\right)$, $K$ is included in $\Omega_{F}$.    The combination of the above estimates leads to
	\begin{equation}
		\norm{\delta_h^{F,i}}_{0,K}+h_K\norm{\nabla\delta_h^{F,i}}_{0,K}\leqslant Ch_K^{3\slash2}.
		\label{scale}
	\end{equation}
	
	\textbf{Step 4 estimates the coefficients in \eqref{delta_h1}.} The coefficient of $\delta_h^{F,i}$ in the definition of $\delta_h^1$ can be estimated by standard arguments. The Cauchy inequality  and the Lagrange basis function of leads to
	\begin{align*}
	\abs{\int_{F}\bm{n}^T\left(\tau-I_h\tau\right)\bm{n}\cdot q_i\dx{S}}\leqslant \norm{\tau-I_h\tau}_{0,F}\norm{q_i}_{0,F}\leqslant|F|^{1/2}\norm{\tau-I_h\tau}_{0,F}.
	\end{align*}
This and the local trace inequality show
\begin{align}	
\label{coef}
\begin{aligned}
\big|\int_{F}\bm{n}^T&\left(\tau-I_h\tau\right)\bm{n}\cdot q_i\dx{S}\big| \\
&\leqslant C|F|^{1/2}h_K^{1/2}(\norm{\nabla(\tau-I_h\tau)}_{0,K}+h_K^{-1}\norm{\tau-I_h\tau}_{0,K}).
\end{aligned}
\end{align}
	Recall the element patch $\Omega_K$. The Scott-Zhang interpolation operator \cite{ScottZhang} satisfies the following approximation property:
	\begin{equation}
		\norm{\nabla\left(\tau-I_h\tau\right)}_{0,K}^2+h_K^{-2}\norm{\tau-I_h\tau}_{0,K}^2\leqslant C\abs{\tau}_{1,\Omega_K}^2.
		\label{SZ}
	\end{equation}
	Since $\abs{F}\geqslant Ch_K^2$, the combination of the inequalities \eqref{coef} and \eqref{SZ} leads to
	\begin{equation}
		\Big|\frac{\int_{F}\bm{n}^T\left(\tau-I_h\tau\right)\bm{n}\cdot q_i\dx{S}}{\abs{F}}\Big|\leqslant Ch_K^{-1\slash2}\abs{\tau}_{1,\Omega_K}.
		\label{ineq: coef}
	\end{equation}

	\textbf{Step 5 concludes the proof of \eqref{delta-h1-bound}.} Given any $K\in\T_h$, $\delta_h^1$ in \eqref{delta_h1} equals
	\begin{equation}
		\trace{\delta_h^1}{K}=\sum_{F'\in\F(\T_h):\ K\subset\Omega_{F'}}\sum_{i=1}^{3}\frac{\int_{F'}\bm{n}^T\left(\tau-I_h\tau\right)\bm{n}\cdot q_i\dx{S}}{\abs{F'}}\delta_h^{F',i}.
	\end{equation}
	This, the inequalities \eqref{scale} and \eqref{ineq: coef} show
	\begin{equation*}
	\begin{aligned}
		\norm{\nabla\delta_h^1}_{0,K}^2+h_K^{-2}\norm{\delta_h^1}_{0,K}^2&\leqslant C\sum_{F'\in\F(\T_h):\ K\subset\Omega_{F'}}\abs{\tau}_{1,\Omega_{K_{F'}}}^2.
	\end{aligned}
	\end{equation*}
	Here $K_{F'}$ is an element containing face $F'$. For any $K\in\T_h$, define a patch $S_K=\bigcup\left\{\Omega_{K'}:\ K'\subset\Omega_K\right\}$. The summation over $K\in\T_h$ leads to
	\begin{equation*}
		\sum_{K\in\T_h}\left(\norm{\nabla\delta_h^1}_{0,K}^2+h_K^{-2}\norm{\delta_h^1}_{0,K}^2\right)\leqslant C\sum_{K\in\T_h}\sum_{K'\subset S_K}\abs{\tau}_{1,K'}^2\leqslant C\abs{\tau}_1^2.
	\end{equation*}
	This concludes the proof.
	\end{proof}

	Since $v\cdot\bm{n}\in P_1(F)$ for any $v\in RM(K)$, $\tau_h=I_h\tau+\delta_h^1$ in Lemma \ref{Lem: delta_h1}  is the desired interpolation satisfying \eqref{eq-normal-normal}. Secondly, the subsequent theorem uses face DoFs in \eqref{DOF1} to modify $\tau_h$ so that \eqref{eq-tangential-normal} holds as well. The results in the literature only hold for stress spaces involving $P_k$ ($k\geqslant4$) polynomials or on some special macro-element meshes.

	\begin{thm}
		\label{step1}
		Under Assumption \ref{ass}, for any $v_h\in V_h$, there exists a $\tau_h\in\Sigma_h$, such that it holds
		\begin{equation*}
			\int_{K}\left(\dv\tau_h-v_h\right)\cdot v\dx{x}=0,\quad\forall v\in RM(K),
		\end{equation*}
		and
		\begin{equation*}
			\norm{\tau_h}_{H(\dv)}\leqslant C\norm{v_h}_0.
		\end{equation*}
	\end{thm}
	\begin{proof}
		From the wellposedness of the continuous problem \eqref{continuousP}, see \cite{Arnold-Winther-conforming} for the 2D case, there exists a $\tau\in H^1\left(\Omega;\Sym\right)$ such that 
		\begin{equation}
			\dv\tau=v_h,\quad\norm{\tau}_1\leqslant C\norm{v_h}_0.
			\label{tau_continuous}
		\end{equation}
		
		Let $I_h\tau$ be the Scott-Zhang interpolation of $\tau$ and $\delta_h^1$ be given by Lemma \ref{Lem: delta_h1}. The third DoFs in \eqref{DOF1} show that  there exists a $\delta_h^2\in\Sigma_{h}$ such that for any $F\in \F(\T\nolimits_h)$, $\bm{n}^T\delta_h^2\bm{n}$ vanishes on $F$ and
		\begin{equation}
			\label{eq: delta_h2}
			\int_{F}\Pi_F\left(\delta_h^2\bm{n}\right)\cdot w\dx{S}=\int_{F}\Pi_F\left(\tau\bm{n}-I_h\tau\bm{n}-\delta_h^1\bm{n}\right)\cdot w\dx{S},\,\forall w\in\operatorname{ND}_1(F).
		\end{equation}
		
		The equivalence of norms implies
		\begin{equation*}
			\sum_{K\in\T_h}\left(\norm{\nabla\delta_h^2}_{0,K}^2+h_K^{-2}\norm{\delta_h^2}_{0,K}^2\right)\leqslant C\sum_{F\in\F\left(\T_h\right)}|F|^{-1/2}\norm{\Pi_F\left(\tau\bm{n}-I_h\tau\bm{n}-\delta_h^1\bm{n}\right)}_{0,F}^2.
		\end{equation*}
		This and the local trace inequality implies
		\begin{equation}
		\begin{aligned}
		&\sum_{K\in\T_h}\left(\norm{\nabla\delta_h^2}_{0,K}^2+h_K^{-2}\norm{\delta_h^2}_{0,K}^2\right)\\
		\leqslant &C\sum_{K\in\T_h}\left(\norm{\nabla\left(\tau-I_h\tau-\delta_h^1\right)}_{0,K}^2+h_K^{-2}\norm{\tau-I_h\tau-\delta_h^1}_{0,K}^2\right).
		\end{aligned}
		\label{scale2}
		\end{equation}

		Let $\tau_h=I_h\tau+\delta_h^1+\delta_h^2$. Lemma \ref{Lem: delta_h1} and \eqref{eq: delta_h2} show that $\tau_h$ satisfies \eqref{eq-normal-normal} and \eqref{eq-tangential-normal}. This guarantees \eqref{modifiedSZ}. Integration by parts and \eqref{tau_continuous} lead to
		\begin{equation*}
				0=\int_{\partial K}(\tau_h-\tau)\bm{n}\cdot v\dx{S}=\int_{K}\dv\left(\tau_h-\tau\right)\cdot v\dx{x}=\int_{K}\left(\dv\tau_h-v_h\right)\cdot v\dx{x}.
		\end{equation*}

		The inequalities \eqref{delta-h1-bound}, \eqref{SZ} and \eqref{scale2} imply
		\begin{equation}
			\begin{aligned}
			&\sum_{K\in\T_h}\left(\norm{\nabla\delta_h^2}_{0,K}^2+h_K^{-2}\norm{\delta_h^2}_{0,K}^2\right)\\
			\leqslant& C\hspace{-2mm}\sum_{K\in\T_h}\hspace{-2mm}\left(\norm{\nabla\left(\tau-I_h\tau\right)}_{0,K}^2+h_K^{-2}\norm{\tau-I_h\tau}_{0,K}^2+\norm{\nabla\delta_h^1}_{0,K}^2+h_K^{-2}\norm{\delta_h^1}_{0,K}^2\right)\hspace{-2mm}\\
			\leqslant &C|\tau|_1^2.
			\end{aligned}
			\label{delta-h2-bound}
		\end{equation}
		Combing the inequalities \eqref{delta-h1-bound}, \eqref{SZ} and \eqref{delta-h2-bound} shows
		\begin{equation*}
			\norm{\tau_h}_{H(\dv)}^2\leqslant\sum_{K\in\T_h}\norm{\tau_h}_{1,K}^2\leqslant\sum_{K\in\T_h}\left(\norm{I_h\tau}_{1,K}^2+\norm{\delta_h^1}_{1,K}^2+\norm{\delta_h^2}_{1,K}^2\right)\leqslant C\norm{\tau}_1^2\leqslant C\norm{v_h}_0^2.
		\end{equation*}
		This concludes the proof.
	\end{proof}

	Define the orthogonal complement $RM^{\perp}(K)=\{w\in P_2(K;\R^3):\ \left(w,v\right)_K=0,\ \forall v\in RM(K)\}$. With Theorem \ref{step1} and \cite[Theorem 2.2]{Hu2015trianghigh}, the following theorem shows that the newly-proposed pair $\Sigma_h-V_h$ in \eqref{stress_space}-\eqref{disp_space} satisfies the discrete inf-sup condition.

	\begin{thm}[Stability]
		\label{inf-sup}Under Assumption \ref{ass},
		the following discrete inf-sup condition holds:
		\begin{equation*}
			\inf_{v_h\in V_h}\sup_{\tau_h\in\Sigma_h}\frac{\int_{\Omega}\dv\tau_h\cdot v_h\dx{x}}{\norm{v_h}_0\norm{\tau_h}_{H(\dv)}}>0.
		\end{equation*}
			\end{thm}
		
		\begin{proof}
			The proof follows the two-step method introduced in \cite{Hu2015trianghigh}.
			
			\textbf{Step 1.} From Theorem \ref{step1}, for any $v_h\in V_h$, there exists a $\tau_h^1\in\Sigma_h$, such that
			\begin{equation*}
				\int_{K}\left(\dv\tau_h^1-v_h\right)\cdot v\dx{x}=0,\quad\forall v\in RM(K),
			\end{equation*}
			and
			\begin{equation*}
				\norm{\tau_h^1}_{H(\dv)}\leqslant C\norm{v_h}_0.
			\end{equation*}

			\textbf{Step 2.} For each element $K$, $v_h-\dv\tau_h^1\in RM^{\perp}(K)$, therefore it follows from \cite[Theorem 2.2]{Hu2015trianghigh}  that there exists a $\tau_h^2\in\Sigma_h$ such that $\trace{\tau_h^2}{K}\in\Sigma_{K,3,b}$ for any $K\in\T_h$ and
			\begin{equation*}
				\dv\tau_h^2=v_h-\dv\tau_h^1,\quad\norm{\tau_h^2}_{H(\dv)}\leqslant C\norm{\dv\tau_h^1-v_h}_{0}.
			\end{equation*}
			
			Set $\tau_h=\tau_h^1+\tau_h^2$. This shows $\dv\tau_h=v_h$ and
			\begin{equation*}
				\norm{\tau_h}_{H(\dv)}\leqslant\norm{\tau_h^1}_{H(\dv)}+C\norm{\dv\tau_h^1-v_h}_{0}\leqslant C\norm{v_h}_0.
			\end{equation*}
This concludes the proof.
		\end{proof}

	\subsection{Error estimates}

	By using the discrete inf-sup condition, similar arguments as in \cite{Hu2015trianghigh,HuZhang2015tre} will lead to the following error estimates for the pair $\Sigma_h-V_h$. The proofs are provided for completeness.

	\begin{thm}[Error estimate]
		Let $(\sigma,u)\in\Sigma\times V$ be the exact solution and $(\sigma_h,u_h)\in\Sigma_h\times V_h$ be the finite element solution.  Under Assumption \ref{ass}, then\begin{equation*}
			\norm{\sigma-\sigma_h}_{H(\dv)}+\norm{u-u_h}_0\leqslant Ch^3\left(\norm{\sigma}_4+\norm{u}_3\right).
		\end{equation*}
	\end{thm}
	\begin{proof}
		The standard theory of the mixed finite element methods give the following quasi-optimal error estimate:
		\begin{equation*}
			\norm{\sigma-\sigma_h}_{H(\dv)}+\norm{u-u_h}_0\leqslant C\inf_{(\tau_h,v_h)\in \Sigma_h\times V_h}\left(\norm{\sigma-\tau_h}_{H(\dv)}+\norm{u-v_h}_0\right).
		\end{equation*}
		
		Recall that the Scott-Zhang interpolation $I_h:\ H^1(\Omega;\Sym)\rightarrow\Sigma_h$ and the $L^2$ projection $P_h:\ L^2(\Omega;\R^3)\rightarrow V_h$ satisfy the following approximation properties:
		\begin{equation*}
		\begin{aligned}
			\norm{\sigma-I_h\sigma}_0+h\abs{\sigma-I_h\sigma}_1\leqslant Ch^4\norm{\sigma}_4,\\
			\norm{u-P_hu}_0\leqslant Ch^3\norm{u}_3.
		\end{aligned}
		\end{equation*}
	Then the proof is completed by choosing $\tau_h=I_h\sigma$ and $v_h=P_hu$.
	\end{proof}

	\begin{thm}[Improved estimates]
		Under Assumption \ref{ass}, the following estimates hold:
		
		\begin{equation*}
		\begin{aligned}
		&\norm{\dv\sigma-\dv\sigma_h}_0\leqslant Ch^r\norm{\dv\sigma}_r,\ 0\leqslant r\leqslant3,\\
		&\norm{\sigma-\sigma_h}_0\leqslant Ch^{r}\norm{\sigma}_{r},\ 1\leqslant r\leqslant4.
		\end{aligned}
		\end{equation*}
	\end{thm}
	\begin{proof}
		For $\tau\in H^1(\Omega;\Sym)$, denote its Scott-Zhang interpolation by $I_h\tau$ and the modifications in Theorem \ref{step1} by $\delta_h^1$ and $\delta_h^2$. Then for any $v\in RM(K)$ it holds
		\begin{equation*}
			\int_{K}\left(P_h\dv\tau-\dv\left(I_h\tau+\delta_h^1+\delta_h^2\right)\right)\cdot v\dx{x}=\int_{K}\dv\left(\tau-I_h\tau-\delta_h^1-\delta_h^2\right)\cdot v\dx{x}=0.
		\end{equation*}
This implies $P_h\dv\tau-\dv\left(I_h\tau+\delta_h^1+\delta_h^2\right)\in RM^{\perp}(K)$. Similar as in Theorem \ref{inf-sup}, this   shows that there exists a piecewise $H(\dv,\Sym)$ bubble function $\tau_h^2\in\Sigma_h$  and $\tau_h^2|_K\in\Sigma_{K,3,b}$ for any $K\in\T_h$ such that
		\begin{equation*}
			\begin{aligned}
				&\dv\tau_h^2=P_h\dv\tau-\dv\left(I_h\tau+\delta_h^1+\delta_h^2\right),\\
				&\norm{\tau_h^2}_{H(\dv)}\leqslant C\norm{P_h\dv\tau-\dv\left(I_h\tau+\delta_h^1+\delta_h^2\right)}_{0}.
			\end{aligned}
		\end{equation*}
		
		Define the following  operator $\Pi_h: H^1(\Omega;\Sym)\rightarrow\Sigma_h$ by setting $\Pi_h\tau=I_h\tau+\delta_h^1+\delta_h^2+\tau_h^2$, which satisfies
		\begin{equation*}
		\left\{
			\begin{aligned}
			&\dv\Pi_h\tau=P_h\dv\tau,\quad\forall\tau\in H^1(\Omega;\Sym),\\
			&\norm{\Pi_h\tau}_{H(\dv)}\leqslant C\norm{\tau}_1,\\
			&\norm{\tau-\Pi_h\tau}_0\leqslant Ch^{r}\norm{\tau}_r,\quad\forall\tau\in H^r(\Omega;\Sym),\ 1\leqslant r\leqslant4.
			\end{aligned}
		\right.
		\end{equation*}
		
		The stability of $\Pi_h$ follows from the $H^1$ stability of the Scott-Zhang interpolation and the norm estimates for $\delta_h^1$ and $\delta_h^2$. The approximation property of $\Pi_h$ follows from the approximation theory of the Scott-Zhang interpolation operator and for $1\leqslant r\leqslant 4$ it holds:
		\begin{equation*}
			\begin{aligned}
				\norm{\delta_h^1}_0&\leqslant C\left(\sum_{K\in\T_h}\left(\norm{\tau-I_h\tau}_{0,K}^2+h_K^2\norm{\nabla\left(\tau-I_h\tau\right)}_{0,K}^2\right)\right)^{1\slash2}\leqslant Ch^r\norm{\tau}_r,\\
				\norm{\delta_h^2}_0&\leqslant C\left(\sum_{K\in\T_h}\left(\norm{\tau-I_h\tau-\delta_h^1}_{0,K}^2+h_K^2\norm{\nabla\left(\tau-I_h\tau-\delta_h^1\right)}_{0,K}^2\right)\right)^{1\slash2}\leqslant Ch^r\norm{\tau}_r,\\
				\norm{\tau_h^2}_0&\leqslant Ch\norm{P_h\dv\tau-\dv\left(I_h\tau+\delta_h^1+\delta_h^2\right)}_{0}\\
				&\leqslant Ch\big(\norm{(I-P_h)\dv\tau}_0+\norm{\dv(I-I_h)\tau}_0+\norm{\dv\delta_h^1}_0+\norm{\dv\delta_h^2}_0\big)\leqslant Ch^{r}\norm{\tau}_r.
			\end{aligned}
		\end{equation*}

		Then the similar argument as in \cite[Theorem 5.1]{Arnold-Winther-conforming} and \cite{Hu2015trianghigh} shows that
		\begin{equation*}
			\begin{aligned}
				&\norm{\dv\sigma-\dv\sigma_h}_0\leqslant Ch^r\norm{\dv\sigma}_r,\ 0\leqslant r\leqslant3,\\
				&\norm{\sigma-\sigma_h}_0\leqslant Ch^{r}\norm{\sigma}_{r},\ 1\leqslant r\leqslant4.
			\end{aligned}
		\end{equation*}
	\end{proof}

	\section{Conclusions}
	
	In this paper, the first  mixed finite element using $P_3$ polynomials for the stress under some mild mesh conditions in Assumption \ref{ass} is proposed to discretize the Hellinger-Reissner formulation of the linear elasticity problem in 3D. The mesh conditions in Assumption \ref{ass} are mild in that if they hold  for the initial mesh, then they hold  for the uniform refinement of the mesh as well. The DoFs of the stress space are established by characterizing a matrix space on the edge patch. The discrete inf-sup condition is analyzed by macro-element techniques and a careful analysis of the first step of the two-step method in \cite{Hu2015trianghigh,HuZhang2015tre}. The DoFs on the edge patches are used to construct locally supported basis functions with respect to the constant and linear moments of the normal-normal component on faces. The future work will be concerned on weakening the mesh conditions. 

\bibliographystyle{siamplain}
\bibliography{ref}

\end{document}